\documentclass[a4paper,leqno,12pt]{amsart}
\usepackage[leqno]{amsmath}
\usepackage{amstext,amssymb,amsthm,enumitem}
\usepackage{graphicx}
\usepackage{pgfplots}
\usepackage{tikz}
\usetikzlibrary{arrows}
\usepackage[cp1250]{inputenc}
\usepackage{float}
\newtheorem{lemma}{Lemma}
\newtheorem{corollary}{Corollary}
\newtheorem{theorem}{Theorem}
\newtheorem{proposition}{Proposition}
\newtheorem{fact}{Fact}
\newtheorem{remark}{Remark}

\newcommand{\NN}{\mathbb{N}}
\newcommand{\QQ}{\mathbb{Q}}
\newcommand{\ZZ}{\mathbb{Z}}
\newcommand{\RR}{\mathbb{R}}
\newcommand{\CCC}{\mathbb{C}}

\newcommand{\BB}{\mathfrak{B}}
\newcommand{\XX}{\mathfrak{X}}
\newcommand{\SSS}{\mathfrak{S}}
\newcommand{\TT}{\mathfrak{T}}

\newcommand{\YY}{\mathfrak{Y}}
\newcommand{\ZZZ}{\mathfrak{Z}}

\newcommand{\cc}{\mathbf{c}}
\newcommand{\CC}{\mathbf{C}}

\newcommand{\AAA}{\mathcal{A}}
\newcommand{\MM}{\mathcal{M}}
\newcommand{\HH}{\mathcal{H}}
\newcommand{\SSSS}{\mathcal{S}}
\newcommand{\TTTT}{\mathcal{T}}

\tikzset{dots/.append style={ultra thick, fill=none}}

\begin{document}
	
	\title[Maximal operators on Lorentz spaces]{Boundedness properties \\ of maximal operators on Lorentz spaces}
	
	\author{Dariusz Kosz}
	\address{ \newline Dariusz Kosz
		\newline Faculty of Pure and Applied Mathematics
		\newline Wroc\l aw University of Science and Technology 
		\newline Wybrze\.ze Wyspia\'nskiego 27, 50-370 Wroc\l aw, Poland
		\newline \textit{E-mail address:} \textnormal{dariusz.kosz@pwr.edu.pl}	
	}
	
	\thanks{This paper constitutes a part of PhD Thesis written in the Faculty of Pure and Applied Mathematics, Wroc\l aw University of Science and Technology, under the supervision of Professor Krzysztof Stempak.	Research was supported by the National Science Centre of Poland, project no. 2016/21/N/ST1/01496.
	}

	\begin{abstract} We study mapping properties of the centered Hardy--Littlewood maximal operator $\MM$ acting on Lorentz spaces. Given $p \in (1,\infty)$ and a metric measure space $\XX$ we let $\Omega^p_{\rm HL}(\XX) \subset [0,1]^2$ be the set of all pairs $(\frac{1}{q},\frac{1}{r})$ such that $\MM$ is bounded from $L^{p,q}(\XX)$ to $L^{p,r}(\XX)$. For each fixed $p$ all possible shapes of $\Omega^p_{\rm HL}(\XX)$ are characterized. Namely, we show that the boundary of $\Omega^p_{\rm HL}(\XX)$ either is empty or takes the form $$\{ \delta \} \times [0, \lim_{u \rightarrow \delta} F(u)] \ \cup \ \{(u, F(u)) : u \in (\delta, 1] \},$$ 
	where $\delta \in [0,1]$ and $F \colon [\delta, 1] \rightarrow [0,1]$ is concave, non-decreasing, and satisfying $F(u) \leq u$. Conversely, for each such $F$ we find $\XX$ such that $\MM$ is bounded from $L^{p,q}(\XX)$ to $L^{p,r}(\XX)$ if and only if the point $(\frac{1}{q}, \frac{1}{r})$ lies on or under the graph of $F$, that is, $\frac{1}{q} \geq \delta$ and $\frac{1}{r} \leq F\big(\frac{1}{q}\big)$. 
	
	
	\end{abstract}
	
	\maketitle
	
	\section{Introduction} \label{S1}
	
	Consider a metric measure space $\XX$, that is, a triple $(X, \rho, \mu)$, where $X$ is a set, $\rho$ is a metric, and $\mu$ is a Borel measure. 
	For any $x \in X$ and $s>0$ by $B(x,s) = B_{\rho}(x,s)$ we denote the open ball centered at $x$ and of radius $s$.  
	According to this the {\it centered Hardy--Littlewood maximal operator}, $\MM_\XX$, is defined by
	\begin{displaymath}
		\MM_\XX f(x) := \sup_{s > 0} \frac{1}{\mu(B(x,s))} \int_{B(x,s)} |f| \, {\rm d}\mu, \qquad x \in X,
	\end{displaymath}
	where $f \colon X \rightarrow \CCC$ is any Borel function. Here and later on, to avoid complicated notation as well as some further pathologies, we assume that balls $B$ satisfying $\mu(B) \in \{0, \infty\}$ are also taken into account, but in each such case we have $\frac{1}{\mu(B)} \int_{B} |f| \, {\rm d}\mu = 0$. 
	
	Recall that an operator $\HH$ is said to be {\it of strong type} $(p,p)$ (resp. {\it of weak type} $(p,p)$) for some $p \in [1, \infty]$ if $\HH$ is bounded on $L^p(\XX)$ (resp. from $L^p(\XX)$ to $L^{p,\infty}(\XX)$). Thus, for example, $\MM_\XX$ is of strong type $(\infty, \infty)$ no matter what the exact structure of $\XX$ is. Moreover, if the measure is {\it doubling} (that is, $\mu(B(x, 2s)) \leq C \mu(B(x, s))$ holds with some $C$ independent of $x$ and $s$), then $\MM_\XX$ is also of weak type $(1,1)$ and hence, by interpolation, of strong type $(p,p)$ for each $p \in (1,\infty)$. However, given an arbitrary (non-doubling) space $\XX$ it may happen that the weak type $(1,1)$ inequality for $\MM_\XX$ fails to occur. For example, in \cite{Sj} Sj\"ogren showed that this is the case for the two-dimensional Gaussian measure ${\rm d}\mu(x,y) = e^{-(x^2+y^2)/2} {\rm d}x {\rm d}y$ and the non-centered Hardy--Littlewood maximal operator (here by ``non-centered'' we mean that the supremum in the definition is taken over the family of balls containing $x$, not only those centered at $x$).
	
	There are several articles devoted to studying various mapping properties of Hardy--Littlewood maximal operators (or their modifications) in the context of non-doubling metric measure spaces (see, e.g., \cite{Al}, \cite{NTV}, \cite{Sa}, \cite{St1}). A particularly interesting task is to find spaces for which such properties are very specific. H.-Q. Li wrote a series of papers \cite{Li1, Li2, Li3} in which the so-called cusp spaces have been introduced for this purpose. For example, in \cite{Li2} it is shown that for each fixed $p_0 \in (1, \infty)$ there exists $\XX$ such that the associated centered maximal operator is of strong type $(p, p)$ if and only if $p > p_0$. 

	Recently, the author also contributed to the development of this field \cite{Ko1, Ko2, Ko3}. In particular, in \cite{Ko3} certain mapping properties of $\MM_\XX$ acting on Lorentz spaces $L^{p,q}(\XX)$ have been studied.
	More precisely, it is proven there that for each $p_0, q_0, r_0 \in (1, \infty)$ with $r_0 \geq q_0$ it is possible to construct 
	
	\begin{enumerate}[label=(\alph*)]
		\item a (non-doubling) space $\XX$ such that $\MM_\XX$ is bounded from $L^{p_0,q_0}(\XX)$ to $L^{p_0,r}(\XX)$ if and only if $r \geq r_0$,
		\item a (non-doubling) space $\XX$ such that $\MM_{\XX}$ is bounded from $L^{p_0,1}(\XX)$ to $L^{p_0,r_0}(\XX)$, but is not bounded from $L^{p_0,q_0}(\XX)$ to $L^{p_0,r_0}(\XX)$.
	\end{enumerate}
	Note that in both cases the boundedness of $\MM_\XX$ acting from $L^{p,q}(\XX)$ to $L^{p,r}(\XX)$ is studied and only one of the parameters $q,r$ is varying, while both, the remaining parameter and $p$, are fixed. The aim of this article is to strengthen the results of \cite{Ko3} by providing a~detailed analysis of the more complex problem in which $p$ is the only fixed parameter and all pairs $(q,r)$ are considered simultaneously. In the following theorem, which is the main result of this article, we characterize all possible shapes of the sets 
	$$
	\Omega^p_{\rm HL}(\XX) := \Big\{\Big(\frac{1}{q},\frac{1}{r}\Big) \in [0,1]^2 : \MM_\XX \ {\rm is} \ {\rm bounded} \ {\rm from} \ L^{p,q}(\XX) \ {\rm to} \ L^{p,r}(\XX)\Big\} \subset [0,1]^2
	$$
	(the shapes of these sets are described in terms of their topological boundaries and the underlying space is the square $[0,1]^2$ with its natural topology). 
	
	\begin{theorem} \label{T0}
		Fix $p \in (1, \infty)$. Then for each $\XX$ one of the following two possibilities holds:
		\begin{itemize}
			\item the boundary of $\Omega^p_{\rm HL}(\XX)$ is empty, that is, $\Omega^p_{\rm HL}(\XX) = \emptyset$ or $\Omega^p_{\rm HL}(\XX) = [0,1]^2$,
			\item the boundary of $\Omega^p_{\rm HL}(\XX)$ is of the form 
			$$\{ \delta \} \times [0, \lim_{u \rightarrow \delta} F(u)] \ \cup \ \{(u, F(u)) : u \in (\delta, 1] \},$$ 
			where $\delta \in [0,1]$ and $F \colon [\delta, 1] \rightarrow [0,1]$ is concave, non-decreasing, and satisfying $F(u) \leq u$.
		\end{itemize}
		Conversely, for each $F$ as above there exists $\XX$ such that 
		%$(\frac{1}{q}, \frac{1}{r}) \in \Omega^p_{\rm HL}(\XX)$ 
		$\MM_\XX$ is bounded from $L^{p,q}(\XX)$ to $L^{p,r}(\XX)$
		if and only if the point $(\frac{1}{q}, \frac{1}{r})$ lies on or under the graph of $F$, that is, $\frac{1}{q} \geq \delta$ and $\frac{1}{r} \leq F\big(\frac{1}{q}\big)$.
	\end{theorem} 
	
	To prove Theorem~\ref{T0} we should focus on two separate tasks. First we want to indicate some conditions that the sets $\Omega^p_{\rm HL}(\XX)$ must satisfy in general, in order to ensure that no situation other than those listed in Theorem~\ref{T0} is possible. This problem is treated in Section~\ref{S5} (see Remark~\ref{R0}, Remark~\ref{R1}, and Theorem~\ref{T2}). The second goal, which turns out to be the harder one, is to introduce a special class of spaces for which we are able to control precisely the behavior of the maximal operator and, at the same time, this behavior is very peculiar. This problem is covered by Theorem~\ref{T1} stated below. We note that, in fact, Theorem~\ref{T1} is slightly more general and it consists of four similar results which have been collected together for the sake of completeness. In what follows, for each $p \in (1, \infty)$ and $q, r \in [1, \infty]$ by $\cc(p,q,r,\XX)$ we mean the smallest constant $C$ for which
	\begin{displaymath}
	\|\MM_\XX f \|_{p,r} \leq C \|f\|_{p,q}, \qquad f \in L^{p,q}(\XX),
	\end{displaymath}
	holds (we use the convention $\cc(p,q,r,\XX) = \infty$ if no such constant exists).
	
	\begin{theorem} \label{T1}
	Fix $p \in (1, \infty)$ and $\delta \in [0, 1]$ (resp. $\delta \in [0, 1)$). Let $F \colon [\delta, 1] \rightarrow [0,1]$ (resp. $F \colon (\delta, 1] \rightarrow [0,1]$) be concave, non-decreasing and satisfying $F(u) \leq u$ for each $u \in [\delta, 1]$ (resp. $u \in (\delta, 1]$). Then
	\begin{itemize}
		\item there exists a (non-doubling) metric measure space $\YY$ such that $\cc(p,q,r, \YY) < \infty$ if and only if $\frac{1}{q} \geq \delta$ (resp. $\frac{1}{q} > \delta$) and $\frac{1}{r} \leq F\big(\frac{1}{q}\big)$, 
		\item there exists a (non-doubling) metric measure space $\ZZZ$ such that $\cc(p,q,r, \ZZZ) < \infty$ if and only if $\frac{1}{q} \geq \delta$ (resp. $\frac{1}{q} > \delta$) and $\frac{1}{r} < F\big(\frac{1}{q}\big)$.
	\end{itemize} 
	\end{theorem}
	
	\noindent A short comment should be made here regarding the spaces $\YY$ and $\ZZZ$. Although the word ``exists'' is used in the formulation of Theorem~\ref{T1}, each space is constructed explicitly. Moreover, the construction process described later on originates in the idea of Stempak, who provided some interesting examples of spaces, when dealing with a certain related problem regarding modified maximal operators (see \cite{St2}).    
	
	The rest of the paper is organized as follows. In Section~\ref{S2} some basic properties of Lorentz spaces are stated and the space combining technique introduced in \cite{Ko3} is recalled. In Section~\ref{S3} and Section~\ref{S4} we study the behavior of the maximal operator in the context of two classes of very specific spaces, namely the so-called test spaces and composite test spaces. The latter class will be used in Section~\ref{S5} to prove Theorem~\ref{T1}. Section~\ref{S6} is devoted to indicating several properties of $\Omega^p_{\rm HL}(\XX)$, which allow us to deduce the first part of Theorem~\ref{T0}. In particular, we formulate a suitable interpolation theorem for Lorentz spaces $L^{p,q}(\XX)$ with the first parameter fixed and the second varying (see Theorem~\ref{T2}). This theorem, in fact, follows from a much more general result \cite[Theorem 5.3.1]{BL} using advanced interpolation methods. However, in the Appendix we give its elementary proof, which, to the author's best knowledge, has never appeared in the literature so far. To avoid misunderstandings, we note that several times in the paper we identify $1 / \infty$ and $1/0$ with $0$ and $\infty$, respectively, when dealing with $q, r \in [1, \infty]$ and $u, F(u) \in [0,1]$. Also, for $\delta = 1$ the conventions $[\delta, 1] = \{1\}$, $(\delta,1] = \emptyset$, and $\lim_{u \rightarrow \delta} F(u) = F(1)$ are used. Finally, we emphasize that, in view of the equality $\MM_\XX f  = \MM_\XX |f|$, we can and will focus only on functions $f \geq 0$.
	
	\subsection*{Acknowledgements}
	I would like to express my deep gratitude to Professor Krzysztof Stempak
	for his valuable remarks and continuous help during the preparation of the paper. I am also indebted to Professor Lech Maligranda who explained to me how Theorem~\ref{T2} in Section~\ref{S6} can be derived from the general theory of interpolation. 
	
	\section{Preliminaries} \label{S2}
	
	We begin with some basic information about Lorentz spaces $L^{p,q}(\XX)$ (for more detailed studies see, e.g., \cite{BS}). For any Borel function $f \colon X \rightarrow \CCC$ we define the {\it distribution function} $d_f \colon [0, \infty) \rightarrow [0, \infty)$ by
	\begin{displaymath}
		d_f(t) := \mu( \{ x \in X : |f(x)| \geq t \}  ).
	\end{displaymath}
	Then for any $p \in [1, \infty)$ and $q \in [1, \infty]$ the space $L^{p,q}(\XX)$ consists of those functions $f$ for which the following quasi-norm 
	\begin{displaymath}
		\|f\|_{p,q} := \left\{ \begin{array}{rl}	
	p^{1/q} \Big( \int_0^\infty \big( t \, d_f(t)^{1/p} \big)^q \frac{{\rm d}t}{t}   \Big)^{1/q}& \textrm{if }  q \in [1, \infty),   \\
			\sup_{t > 0} t \, d_f(t)^{1/p} & \textrm{if }  q = \infty, \end{array} \right. 
	\end{displaymath}
	is finite. Recall that if $p=q$, then $(L^{p,q}(\XX), \| \cdot \|_{p,q})$ coincides with the standard Lebesgue space $(L^p(\XX), \| \cdot \|_p)$. Now we present several facts concerning $L^{p,q}(\XX)$ spaces. The metric measure space is arbitrary here, except for the condition $\mu(X) < \infty$ assumed in Fact~\ref{F2}.  
	
	\begin{fact} \label{F1}
		Let $p \in (1, \infty)$, $q \in [1, \infty]$, and $n_0 \in \NN$. Then there exists a numerical constant $\CC_\triangle(p,q)$ independent of $n_0$ and $\XX$ such that
		\begin{displaymath}
		\big\| \sum_{n=1}^{n_0} f_n \big\|_{p,q} \leq \CC_{\triangle}(p,q) \sum_{n=1}^{n_0} \|f_n \|_{p,q}, \qquad f_n \in L^{p,q}(\XX), \ n \in \{1, \dots, n_0\}. 
		\end{displaymath} 
	\end{fact}
	
	\begin{fact} \label{F2}
		Let $p \in (1, \infty)$ and $q \in [1, \infty]$, and assume that $\mu(X) < \infty$. Then there exists a~numerical constant $\CC_{\rm{avg}}(p,q)$ independent of $\XX$ such that
		\begin{displaymath}
		\| f_{\rm{avg}} \|_{p,q} \leq \CC_{\rm{avg}}(p,q) \|f \|_{p,q}, \qquad f \in L^{p,q}(\XX), 
		\end{displaymath}
		where $ f_{\rm{avg}} := \|f\|_1 / \mu(\XX)$ is constant. 
	\end{fact}
	
	\begin{fact} \label{F3}
		Let $p \in (1, \infty)$ and $1 \leq q \leq r \leq \infty$. Then $L^{p,q}(\XX) \subset L^{p,r}(\XX)$ and there exists a~numerical constant $\CC_{\hookrightarrow}(p,q,r)$ independent of $\XX$ such that
		\begin{displaymath}
		\|f \|_{p,r} \leq \CC_{\hookrightarrow}(p,q,r) \|f \|_{p,q}, \qquad f \in L^{p,q}(\XX). 
		\end{displaymath} 
	\end{fact}
	
	\noindent Fact~\ref{F3} is well known (see, e.g., \cite[Proposition 4.2]{BS}), while Fact~\ref{F1} and Fact~\ref{F2} are easy consequences of \cite[Lemma 4.5 and Theorem 4.6]{BS}.
		
		We also need the following auxiliary lemma.

		\begin{lemma} \label{L1}
			Fix $p \in (1, \infty)$, $q \in [1, \infty]$, and $n_0 \in \NN$, and consider a finite sequence of functions $\{ f_n \}_{n=1}^{n_0}$ with pairwise disjoint supports $A_n$. Assume that for each $n \geq 2$ and $t > 0$ we have either $d_{f_n}(t) \geq \mu(A_1 \cup \dots \cup A_{n-1})$ or $d_{f_n}(t) = 0$. Then for some numerical constant $\CC_1 = \CC_1(p,q)$ independent of $\XX$, $n_0$, and $\{ f_n \}_{n=1}^{n_0}$ we have: if $q \in [1, \infty)$,
			\begin{displaymath}
			\frac{1}{\CC_1} \, \Big(  \sum_{n=1}^{n_0} \| f_n\|_{p,q}^q \Big)^{1/q} \leq \big\| \sum_{n=1}^{n_0} f_n \big\|_{p,q} \leq \CC_1 \, \Big(  \sum_{n=1}^{n_0} \| f_n\|_{p,q}^q \Big)^{1/q},
			\end{displaymath}
			or, if $q = \infty$,
			\begin{displaymath}
			\frac{1}{\CC_1} \, \sup_{n \in \{1, \dots, n_0\}} \| f_n\|_{p,\infty} \leq \big\| \sum_{n=1}^{n_0} f_n \big\|_{p,\infty} \leq \CC_1 \, \sup_{n \in \{1, \dots, n_0\}} \| f_n\|_{p,\infty}. 
			\end{displaymath}
		\end{lemma}
		
		\begin{proof}
			Let $f = \sum_{n=1}^{n_0} f_n$ and consider $q \in [1, \infty)$ (the case $q = \infty$ is very similar). The thesis is an easy consequence of the fact that, under the specified assumptions, the quantities $d_{f}(t)^{1/p}$ and $( \sum_{n=1}^{n_0} d_{f_n}(t)^{q/p})^{1/q}$ are comparable to each other with multiplicative constants independent of $t > 0$.    
		\end{proof}
	
	The main tool used in the proof of Theorem~\ref{T1} is the {\it space combining technique}, which, in the context of Lorentz spaces, has been introduced in \cite{Ko3}. Here we present only the key result that the application of this technique gives.
	
	\begin{proposition} \label{P1}
		Let $(\XX_n)_{n \in \NN}$ be a given sequence of metric measure spaces and assume that each of them consists of finitely many elements. Let $\XX$ be the space constructed with an aid of $(\XX_n)_{n \in \NN}$ by using the method described in \cite[Section 4]{Ko3}. Then for each $p \in (1, \infty)$ and $1 \leq q \leq r \leq \infty$ we have 
		\begin{equation} \label{eq1}
			\frac{1}{\CC} \, \sup_{n \in \NN} \, \cc(p,q,r,\XX_n) \leq \cc(p,q,r,\XX) \leq \CC \,\sup_{n \in \NN} \, \cc(p,q,r,\XX_n),
		\end{equation}
		where $\CC = \CC(p,q,r)$ is a numerical constant independent of $(\XX_n)_{n \in \NN}$.
	\end{proposition}
	
	\noindent Proposition~\ref{P1} can be deduced directly from the proof of \cite[Proposition 1]{Ko3}. We also notice that the two most important ingredients used to obtain \eqref{eq1} are Fact~\ref{F2} and a~certain argument in the spirit of Lemma~\ref{L1}. 
	
	Two more comments are in order here. Whenever we want to apply Proposition~\ref{P1} later on, we omit the details related to the proper indexing of the component spaces. The only important fact is that each time we use countably many spaces. Finally, we indicate that each space $\XX$ obtained by using Proposition~\ref{P1} is non-doubling. 
	
	\section{Test spaces} \label{S3}
	
	We now introduce and analyze certain auxiliary structures which we call \textit{test} \textit{spaces} later on. We emphasize here that each test space may be used as a component space in Proposition~\ref{P1}, because it consists of finitely many elements.
	
	Let $p \in (1, \infty)$ and take $N, M, L \in \NN$. We associate with each quadruple $(p, N, M, L)$ four sequences of positive integers, $(m_i)_{i=1}^N$, $(h_i)_{i=1}^N$, $(\alpha_i)_{i=1}^M$, and $(\beta_i)_{i=1}^M$, satisfying the following assertions:
		\begin{enumerate}[label=(\roman*)]
			\item \label{i} $h_{N} / h_i \in \NN$,
			\item \label{ii} $m_{i+1} \geq 2 m_i h_i$,
			\item \label{iii} $1 \leq m_{i}^{1-p}h_i <2$,
			\item \label{iv} $\alpha_1 \geq 2 m_N h_N$,
			\item \label{v} $\alpha_{i+1} \geq 2 \alpha_i L \beta_i h_N$,
			\item \label{vi} $1 \leq \alpha_{i}^{1-p} \beta_i h_N <2$.	
		\end{enumerate}
	We kindly ask the reader to consult \cite{Ko3} in order to make sure that the properties~\ref{i}--\ref{vi} can be met simultaneously. Some further explanations about the role of these objects are also given there.
	
	For fixed $K \in [1, \infty)$ we define a test space $\SSS = \SSS_{p,N,M,K,L} = (S, \rho, \mu)$ as follows. Set
	\begin{displaymath}
	S := \{x_{i,j}, \, x^\circ_{k,l} : i=1, \dots, N, \ j=1, \dots , h_i, \ k=1, \dots, M, \ l = 1, \dots , L \beta_k h_N\},
	\end{displaymath}
	where all elements $x_{i,j}, \, x^\circ_{k,l}$ are pairwise different. We use some auxiliary symbols for certain subsets of $S$: 
	\begin{displaymath}
	S^\circ := \{x^\circ_{k,l} : k=1, \dots, M, \ l=1, \dots, L \beta_k h_N\};
	\end{displaymath}	
	for $i = 1, \dots, N$ and $k = 1, \dots, M$,
	\begin{displaymath}
	S_{i} := \{x_{i,j} : j=1, \dots, h_i \}, \quad
	S^\circ_{k} := \{x^\circ_{k,l} : l=1, \dots, L \beta_k h_N\};
	\end{displaymath}
	for $i = 1, \dots, N$, $j=1, \dots, h_i$, and $k = 1, \dots, M$,  
	\begin{displaymath}
	S^\circ_{i,j,k} := \Big\{x^\circ_{k,l} : l \in \Big(\frac{j-1}{h_i} L \beta_k h_N, \frac{j}{h_i} L \beta_k h_N \Big]\Big\}.
	\end{displaymath}
	Observe that the sets $S^\circ_{i,j,k}$, $j = 1, \dots ,h_i$, are pairwise disjoint and, in view of \ref{i}, each of them contains exactly $L \beta_{k} h_N  / h_i$ elements. Moreover, $\bigcup_{j=1}^{h_i} S^\circ_{i,j,k} = S^\circ_{k}$ holds for each $i \in \{1, \dots, N\}$. 
	
	We introduce $\mu$ by letting $\mu(\{ x_{i,j} \}) := m_i$ and $\mu(\{ x^\circ_{k,l} \}) := K \alpha_k$. Note that, in view of \ref{iv}, \ref{ii}, and \ref{v}, the following inequalities hold: for each $x \in S^\circ$,
	\begin{displaymath}
	\mu(\{x\}) > \mu(S \setminus S^\circ),
	\end{displaymath}
	for each $1 < i \leq N$ and $x \in S_{i}$,   
	\begin{displaymath}
	\mu(\{x\}) > \mu(S_1 \cup \ldots \cup S_{i-1}),
	\end{displaymath}
	and for each $1 < k \leq M$ and $x^\circ \in S^\circ_{k}$,
	\begin{displaymath}
	\mu(\{x^\circ\}) > \mu(S^\circ_1 \cup \ldots \cup S^\circ_{k-1}).
	\end{displaymath}
	
	Finally, we define $\rho$ by the formula
	\begin{displaymath}
	\rho(x,y) := \left\{ \begin{array}{rl}
	0 & \textrm{if } x=y, \\
	1 & \textrm{if } \{x, y\} = \{x_{i,j},x^\circ_{k,l}\} \textrm{ and } x^\circ_{k,l} \in S^\circ_{i,j,k},  \\
	2 & \textrm{otherwise.} \end{array} \right. 
	\end{displaymath}
	It is worth noting here that for each $i \in \{1 \dots N \}$, $ k \in \{1 \dots M \}$, and $x^\circ \in S^\circ_{k}$, there is exactly one point $x \in S_i$ such that $\rho(x,x^\circ)=1$. This point is denoted by $\Gamma_i(x^\circ)$ later on.
	
	Figure 1 shows a model of the space $(S, \rho)$ for $N=3$ and $M=2$. The solid line between two points indicates that the distance between them equals $1$. Otherwise the distance equals $2$.  
	
	\begin{figure}[H]
		\begin{tikzpicture}
		[scale=.7,auto=left,every node/.style={circle,fill,inner sep=2pt}]
	
		\node[label={[yshift=-1cm]$x_{1,1}$}] (n1) at (1,1) {};
		
		\node[label={[yshift=-1cm]$x_{2,1}$}] (n2) at (5,1) {};
		\node[label={[yshift=-1cm]$x_{2,2}$}] (n3) at (7,1) {};
		
		\node[label={[yshift=-1cm]$x_{3,1}$}] (n4) at (11,1) {};
		\node[label={[yshift=-1cm]$x_{3,2}$}] (n5) at (12.5,1) {};
		\node[label={[yshift=-1cm]$x_{3,3}$}] (n6) at (14,1) {};
		\node[label={[yshift=-1cm]$x_{3,4}$}] (n7) at (15.5,1) {};
		
		\node[label={$x^\circ_{1,1}$}] (m1) at (2,5) {};
		\node[label={$x^\circ_{1,2}$}] (m2) at (3,5) {};
		\node[label={$x^\circ_{1,3}$}] (m3) at (4,5) {};
		\node[label={$x^\circ_{1,4}$}] (m4) at (5,5) {};
		
		\node[label={$x^\circ_{2,1}$}] (m5) at (7.5,5) {};
		\node[label={$x^\circ_{2,2}$}] (m6) at (8.5,5) {};
		\node[label={$x^\circ_{2,3}$}] (m7) at (9.5,5) {};
		\node[label={$x^\circ_{2,4}$}] (m8) at (10.5,5) {};
		\node[label={$x^\circ_{2,5}$}] (m9) at (11.5,5) {};
		\node[label={$x^\circ_{2,6}$}] (m10) at (12.5,5) {};
		\node[label={$x^\circ_{2,7}$}] (m11) at (13.5,5) {};
		\node[label={$x^\circ_{2,8}$}] (m12) at (14.5,5) {};
		
		\foreach \from/\to in {n1/m1, n1/m2, n1/m3, n1/m4, n1/m5, n1/m6, n1/m7, n1/m8, n1/m9, n1/m10, n1/m11, n1/m12}
		\draw (\from) -- (\to);
		
		\foreach \from/\to in {n2/m1, n2/m2, n3/m3, n3/m4, n2/m5, n2/m6, n2/m7, n2/m8, n3/m9, n3/m10, n3/m11, n3/m12}
		\draw (\from) -- (\to);
		
		\foreach \from/\to in {n4/m1, n5/m2, n6/m3, n7/m4, n4/m5, n4/m6, n5/m7, n5/m8, n6/m9, n6/m10, n7/m11, n7/m12}
		\draw (\from) -- (\to);
		
		\end{tikzpicture}
		\caption{The model of the space $(S, \rho)$ for $N=3$ and $M=2$.}
	\end{figure}
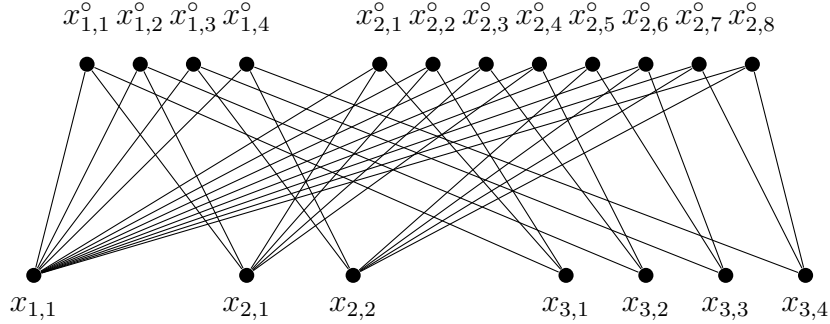
	
	For the convenience of the reader we explicitly describe any ball $B \subset S$. Thus we have: for $i \in \{ 1, \dots, N\}$, $j \in \{1, \dots, h_i\}$, 
	\begin{displaymath}
		B(x_{i,j},s) = \left\{ \begin{array}{rl}
		\{x_{i,j}\} & \textrm{for } 0 < s \leq 1, \\
		\{x_{i,j}\} \cup \{x^\circ \in S^\circ : \Gamma_i(x^\circ) = x_{i,j} \} & \textrm{for } 1 < s \leq 2,  \\
		S & \textrm{for } 2 < s, \end{array} \right.
	\end{displaymath} 
	\noindent and, for $k \in \{1, \dots, M\}$, $l \in \{1, \dots, L\beta_k h_N\}$,
	\begin{displaymath}
	B(x^\circ_{k,l},s) = \left\{ \begin{array}{rl}
	\{x^\circ_{k,l}\} & \textrm{for } 0 < s \leq 1, \\
	\{x^\circ_{k,l}\} \cup \{\Gamma_i(x^\circ_{k,l}) : i = 1, \dots, N \} & \textrm{for } 1 < s \leq 2,  \\
	S & \textrm{for } 2 < s. \end{array} \right.
	\end{displaymath}
	
	Now, for each fixed $i \in \{1 \dots N \}$ and $k \in \{1 \dots M \}$, we introduce a linear operator $\AAA_{k,i} = \AAA_{k,i, \SSS}$ given by the formula
	\begin{displaymath}
	\AAA_{k,i}f(x) := \left\{ \begin{array}{rl}
	\frac{f(\Gamma_i(x)) \, \mu(\{\Gamma_i(x)\})}{\mu(\{x\}) } & \textrm{if } x \in S^\circ_k,  \\
	0 & \textrm{otherwise.} \end{array} \right. 
	\end{displaymath}
	In the following lemma we estimate the norm of $\AAA_{k,i}$ considered as an operator acting from $L^{p,q}(\SSS)$ to $L^{p,r}(\SSS)$.  
	
	\begin{lemma} \label{L2}
		Let $\SSS = \SSS_{p,N,M,K,L}$ be the metric measure space defined as above. Fix $1 \leq q \leq r \leq \infty$, $i \in \{1, \dots, N\}$, and $k \in \{1, \dots M\}$, and consider the operator $\AAA_{k,i}$. Then there exists a numerical constant $\CC_2 = \CC_2(p, q, r)$ independent of $N$, $M$, $K$, $L$, $i$, and $k$ such that
	\begin{displaymath}
	\frac{1}{\CC_2} \, K^{-1+1/p} L^{1/p} \leq \|\AAA_{k,i} \|_{L^{p,q}(\SSS) \rightarrow L^{p,r}(\SSS)} \leq \CC_2 \, K^{-1+1/p} L^{1/p}.
	\end{displaymath}
	\end{lemma}
	
	\begin{proof}
		First we estimate $\|\AAA_{k,i} \|_{L^{p,q}(\SSS) \rightarrow L^{p,r}(\SSS)}$ from above. Take $f \in L^{p,q}(\SSS)$. Since $\AAA_{k,i} f = \AAA_{k,i} (f \cdot \chi_{S_i})$, we can assume that the support of $f$ is contained in $S_i$ (here and anywhere else in the paper $\chi_E$ denotes the indicator function of a given Borel set $E$). If this is the case, then for each $t>0$ we have the equality
	\begin{displaymath}
	d_{\AAA_{k,i}f}(t) = \frac{KL\alpha_k \beta_k h_N}{m_i h_i} d_f(tK\alpha_k / m_i)
	\end{displaymath}  
	and a simple calculation gives
	\begin{displaymath}
	\|\AAA_{k,i}f \|_{p,r} = K^{-1 + 1/p} L^{1/p} m_i^{1-1/p} h_i^{-1/p} \alpha_k^{-1 + 1/p} \beta_k^{1/p} h_N^{1/p} \|f\|_{p,r}.
	\end{displaymath}
	Thus, in view of \ref{iii}, \ref{vi}, and Fact~\ref{F3}, we obtain
	\begin{displaymath}
	\|\AAA_{k,i}f \|_{p,r} \leq 4 \, \CC_{\hookrightarrow}(p,q,r) \, K^{-1 + 1/p} L^{1/p} \|f\|_{p,q}.
	\end{displaymath}
	
	Finally, consider $g = \chi_{S_i}$. Then we have $\AAA_{k,i} g = \frac{m_i}{K\alpha_k} \, \chi_{S^\circ_k}$ and hence
	\begin{displaymath}
	\frac{ \|\AAA_{k,i}g\|_{p,r} }{\| g \|_{p,q} } = K^{-1 + 1/p} L^{1/p} m_i^{1-1/p} \alpha_k^{-1 + 1/p} \beta_k^{1/p} h_N^{1/p} h^{-1/p} \geq K^{-1 + 1/p} L^{1/p}, 
	\end{displaymath} 
	where in the last inequality we again used \ref{iii} and \ref{vi}. 
	\end{proof}

Next we introduce a linear operator $\AAA = \AAA_\SSS$ given by the formula
\begin{displaymath}
\AAA f(x) := \left\{ \begin{array}{rl}
\sum_{i=1}^{N} \AAA_{k,i}f(x) & \textrm{if } x \in S^\circ_k, \, k \in \{1, \dots, M\}, \\
0 & \textrm{otherwise.} \end{array} \right. 
\end{displaymath}
As before, we estimate the norm of $\AAA$ acting from $L^{p,q}(\SSS)$ to $L^{p,r}(\SSS)$.

\begin{lemma} \label{L3}
	Let $\SSS = \SSS_{p,N,M,K,L}$ be the metric measure space defined as above. Fix $1 \leq q \leq r \leq \infty$ and consider the operator $\AAA$. Then there exists a numerical constant $\CC_3 = \CC_3(p, q, r)$ independent of $N$, $M$, $K$, and $L$ such that
	\begin{displaymath}
	\frac{1}{\CC_3} \, K^{-1+1/p} L^{1/p} M^{1/r} N^{1 - 1/q} \leq \|\AAA \|_{L^{p,q}(\SSS) \rightarrow L^{p,r}(\SSS)} \leq \CC_3 \, K^{-1+1/p} L^{1/p} M^{1/r} N^{1 - 1/q}.
	\end{displaymath}
\end{lemma}

\begin{proof}
	First we estimate $\|\AAA \|_{L^{p,q}(\SSS) \rightarrow L^{p,r}(\SSS)}$ from above. Take $f \in L^{p,q}(\SSS)$. Since $\AAA f = \AAA(f \cdot \chi_{S \setminus S^\circ})$, we can assume that the support of $f$ is contained in $S \setminus S^\circ$. We decompose $f = \sum_{i=1}^N f_i$, where $f_i = f \cdot \chi_{S_i}$. Then, by \ref{ii}, \ref{v}, and Lemma~\ref{L1}, we have
	\begin{displaymath}
	\|f\|_{p,q} \geq \frac{1}{\CC_1(p,q)} \, \Big(  \sum_{i=1}^N \| f_i\|_{p,q}^q \Big)^{1/q}
	\end{displaymath}
	and
	\begin{displaymath}
	\|\AAA f\|_{p,r} \leq \CC_1(p,r) \, \Big(  \sum_{k=1}^M \big\| \AAA f \cdot \chi_{S^\circ_k}\big\|_{p,r}^r \Big)^{1/r}.
	\end{displaymath}
	Moreover, by using Fact~\ref{F1} and Lemma~\ref{L2}, we obtain the following estimate
	\begin{displaymath}
	\| \AAA f \cdot \chi_{S^\circ_k}\|_{p,r} \leq \CC_\triangle(p,r) \sum_{i=1}^N \| \AAA_{k,i} f_i \|_{p,r} \leq \CC_\triangle(p,r) \CC_2(p,q,r) K^{-1+1/p} L^{1/p} \sum_{i=1}^N \| f_i \|_{p,q}
	\end{displaymath}
	for each $k \in \{1, \dots, M\}$. Therefore, 
	\begin{displaymath}
	\|\AAA f\|_{p,r} \leq \CC_1(p,r) \CC_\triangle(p,r) \CC_2(p,q,r) K^{-1+1/p} L^{1/p} M^{1/r} \sum_{i=1}^N \| f_i\|_{p,q}.
	\end{displaymath}
	On the other hand, an application of H\"older's inequality gives
	\begin{displaymath}
	\Big(  \sum_{i=1}^N \| f_i\|_{p,q}^q \Big)^{1/q} \geq N^{-1 + 1/q} \sum_{i=1}^N \| f_i\|_{p,q}.
	\end{displaymath}
	Combining the above estimates we conclude that
	\begin{displaymath}
	\|\AAA f\|_{p,r} \leq \CC_1(p,q) \CC_1(p,r) \CC_\triangle(p,r) \CC_2(p,q,r) K^{-1+1/p} L^{1/p} M^{1/r} N^{1 - 1/q} \| f\|_{p,q}.
	\end{displaymath}
	
	Finally, consider $g = \sum_{i=1}^N (h_i m_i)^{-1/p} \cdot \chi_{S_i}$. Then, by using \ref{iii}, we have
	\begin{displaymath}
	\AAA g \geq \sum_{k =1}^M \frac{N}{2^{1/p} K \alpha_k} \cdot \chi_{S^\circ_k}
	\end{displaymath}
	and thus
	\begin{displaymath}
	\frac{ \|\AAA g\|_{p,r} }{\| g \|_{p,q} } \geq \frac{ \Big(\sum_{k=1}^M \big(K^{-1+1/p} L^{1/p} N \alpha_k^{-1+1/p} \beta_k^{1/p} h_N^{1/p} \big)^r \Big)^{1/r}  }{2^{1/p} \, \CC_1(p,q) \CC_1(p,r) N^{1/q}} \geq \frac{ K^{-1+1/p} L^{1/p} M^{1/r} N^{1 - 1/q} }{2^{1/p} \, \CC_1(p,q) \CC_1(p,r)}, 
	\end{displaymath}
	where we used \ref{ii}, \ref{v}, and Lemma~\ref{L1} in the first inequality, and \ref{vi} in the second.
\end{proof}

In the following lemma we estimate the norm of the maximal operator $\MM_\SSS$ acting from $L^{p,q}(\SSS)$ to $L^{p,r}(\SSS)$. This is the main result of this section. 

\begin{lemma} \label{L4}
	Let $\SSS = \SSS_{p,N,M,K,L}$ be the metric measure space defined as above. Fix $1 \leq q \leq r \leq \infty$ and consider the associated operator $\MM_\SSS$. Then there exists a numerical constant $\CC_4 = \CC_4(p, q, r)$ independent of $N$, $M$, $K$, and $L$ such that
	\begin{displaymath}
	\frac{1}{\CC_4} \, \Big( 1 + K^{-1+1/p} L^{1/p} M^{1/r} N^{1 - 1/q} \Big) \leq \cc(p,q,r,\SSS) \leq \CC_4 \, \Big( 1 + K^{-1+1/p} L^{1/p} M^{1/r} N^{1 - 1/q} \Big).
	\end{displaymath}
\end{lemma}

\begin{proof}
First we estimate $\cc(p,q,r,\SSS)$ from above. Take $f \in L^{p,q}(\SSS)$ such that $\|f\|_{p,q} = 1$. It is easy to check that
\begin{displaymath}
\MM_{\SSS} f \leq \max \{f, 4 \AAA f, 2 \widetilde{\MM} f, f_{\rm{avg}}\},
\end{displaymath}
where $\widetilde{\MM} f := \chi_{S \setminus S^\circ} \cdot \max_{x^\circ \in S^\circ} f(x^\circ)$. Therefore, by Fact~\ref{F1}, we have
\begin{displaymath}
\|\MM_{\SSS} f\|_{p,r} \leq 4 \, \CC_\triangle(p,r) \, \Big( \|f\|_{p,r} + \|\AAA f \|_{p,r} + \| \widetilde{\MM} f\|_{p,r} + \| f_{\rm{avg}} \|_{p,r} \Big).
\end{displaymath}
The inequalities $\| \widetilde{\MM} f\|_{p,r} \leq \|f\|_{p,r}$ and $\| f_{\rm{avg}} \|_{p,r} \leq \CC_{\rm{avg}}(p,r) \|f\|_{p,r}$ follows from \ref{iv} and Fact~\ref{F2}, respectively. Combining the above estimates with Lemma~\ref{L3} we conclude that $\|\MM_{\SSS} f\|_{p,r}$ is controlled by
\begin{align*}
4 \, \CC_\triangle(p,r) \, \Big( \CC_{\hookrightarrow}(p,q,r) \big(2 + \CC_{\rm{avg}}(p,r) \big) + \CC_3(p,q,r) \big(K^{-1+1/p} L^{1/p} M^{1/r} N^{1 - 1/q} \big) \Big).
\end{align*}

Now we estimate $\cc(p,q,r,\SSS)$ from below. First, we have $\cc(p,q,r,\SSS) \geq p^{1/r - 1/q} r^{-1/r} q^{1/q}$, since $\|\MM_\SSS g \|_{p,r} = \|g \|_{p,r} =  p^{1/r - 1/q} r^{-1/r} q^{1/q} \, \|g \|_{p,q}$ for $g = \chi_{S}$ (here we use the convention $\infty^{1/\infty} = \infty^{-1/\infty} = 1$, if necessary). Finally, the inequality 
\begin{displaymath}
\cc(p,q,r,\SSS) \geq \frac{1}{\CC_3(p,q,r)} \, K^{-1+1/p} L^{1/p} M^{1/r} N^{1 - 1/q}
\end{displaymath}
is a consequence of Lemma~\ref{L3} and the fact that $\MM_\SSS f \geq \AAA f$ for each $f \in L^{p,q}(\SSS)$.  
\end{proof}

At the end of this section we reformulate the result of the previous lemma in a way that makes it easier to use later on.

\begin{corollary} \label{C1}
	Fix $p \in (1, \infty)$, $\lambda \in (0, \infty)$, and $a, b, \kappa \in \NN$. Let $\SSS_{(p, \lambda, a, b, \kappa)}$ be the test space $\SSS_{p,N,M,K,L}$ with $p$ as above, $N = \kappa^b$, $M = \kappa^a$, and some $K, L$ satisfying $K^{-1 + 1/p} L^{1/p} = \lambda \kappa^{-b}$. Then for each $1 \leq q \leq r \leq \infty$ we have
	\begin{displaymath}
	\frac{1}{\CC_4} \, \Big( 1 + \lambda \kappa^{a/r - b/q}\Big) \leq \cc(p,q,r,\SSS_{(p, \lambda, a, b, \kappa)}) \leq \CC_4 \, \Big( 1 + \lambda \kappa^{a/r - b/q} \Big),
	\end{displaymath}
	where $\CC_4 = \CC_4(p,q,r)$ is the constant from Lemma~\ref{L4}.
\end{corollary}

\section{Composite test spaces} \label{S4}
 In the following two sections by a \textit{composite test space} we mean any metric measure space $\TT$ that arises as a result of applying Proposition~\ref{P1} to a certain family of test spaces introduced in Section~\ref{S3}. This is a bit imprecise, but one can think of composite test spaces as intermediate objects between test spaces and the spaces we want to obtain in Theorem~\ref{T1}. More precisely, these latter ones will be composite test spaces constructed with an aid of a sequence of simpler composite test spaces. We now briefly explain the details of such a construction.
 
 \begin{proposition} \label{P2}
 	Let $(\TT_n)_{n \in \NN}$ be a given sequence of composite test spaces. Then there exists a composite test space $\TT$ such that for each $p \in (1, \infty)$ and $1 \leq q \leq r \leq \infty$ we have
 	\begin{displaymath}
 	\frac{1}{\CC^2} \, \sup_{n \in \NN} \, \cc(p,q,r,\TT_n) \leq \cc(p,q,r,\TT) \leq \CC^2 \,\sup_{n \in \NN} \, \cc(p,q,r,\TT_n),
 	\end{displaymath}  
 	where $\CC = \CC(p,q,r)$ is the constant from Proposition~\ref{P1}. 
 \end{proposition}
 
 \begin{proof}
 Note that each space $\TT_n$ is constructed with an aid of some sequence of test spaces, say $\{\SSS_{n,m} : m \in \NN\}$. We let $\TT$ be the space constructed by using Proposition~\ref{P1} for the whole family of test spaces $\{ \SSS_{n,m} : n,m \in \NN \}$. It follows directly from Proposition~\ref{P1} that $\TT$ satisfies the desired condition.  	
 \end{proof}
 
 Now we will construct some composite test spaces for which the associated maximal operators have very specific properties.
 
 \begin{lemma} \label{L5}
 Let $p \in (1, \infty)$, $\gamma \in \RR$, $a, b, R \in \NN$, and $\epsilon > 0$. Then there exists a composite test space $\TT = \TT_{p, \gamma, a, b, R, \epsilon}$ such that for each $1 \leq q \leq r \leq \infty$ we have
 \begin{align*}
  \cc(p,q,r,\TT) = \infty \quad & \textrm{if } \quad a/r - b/q = \gamma,  \\
  \CC_5^{-1} R^{\epsilon d} \leq \cc(p,q,r, \TT) \leq \CC_5 \big( 1 + R^{2\epsilon d} \big) \quad & \textrm{if } \quad a/r - b/q \in (\gamma - 2 \epsilon d, \gamma - \epsilon d), \\
  \cc(p,q,r,\TT) \leq  \CC_5 \quad & \textrm{if } \quad a/r - b/q \leq \gamma - 3 \epsilon d,
 \end{align*} 
 where $d = \sqrt{a^2 + b^2}$ and $\CC_5 = \CC_5(p,q,r)$ is independent of $\gamma$, $a$, $b$, $R$, and $\epsilon$.	
 \end{lemma}
 
 \noindent Figure 2 describes the behavior of the function $\cc(p, q, r, \TT)$. We notice that the parameter $d$ appears here only for purely aesthetic reasons (for example, the Euclidean distance between the lines $a/r - b/q = \gamma$ and $a/r - b/q = \gamma - \epsilon d$ equals $\epsilon$).
 
 \begin{figure}[H]
 	\begin{tikzpicture}[
 	axis/.style={very thick, ->, >=stealth'},
 	important line/.style={thick},
 	dashed line/.style={dashed, thin},
 	pile/.style={thick, ->, >=stealth', shorten <=2pt, shorten
 		>=2pt},
 	every node/.style={color=black}
 	]
 	
 	\draw[axis] (0,0)  -- (5.5,0) node(xline)[right]
 	{$1/q$};
 	\draw[axis] (0,0) -- (0,5.5) node(yline)[left] {$1/r$};
 	
 	\draw (0.0,0.0) node[below] {$0$} -- (0.0,0.0);
 	\draw (0,4.5) node[left] {$1$} -- (0.1,4.5);
 	\draw (4.5,0) node[below] {$1$} -- (4.5,0.1);
 	
 	\draw[important line] (0,4.5) -- (4.5,4.5);
 	\draw[important line] (4.5,0) -- (4.5,4.5);
 	
 	\draw[dashed line] (0,0) -- (4.5,4.5);
 	
 	\draw[dashed line] (1.5,0) -- (4.5,4);
 	\draw[dashed line] (1.95,0) -- (4.5,3.4);
 	\draw[dashed line] (2.4,0) -- (4.5,2.8);
 	\draw[dashed line] (2.85,0) -- (4.5,2.2);
 	
 	\draw (4.5,4.5) node[right] {$ q=r $} -- (4.5,4.5);
 	
 	\draw (4.5,4) node[right] {$ a/r-b/q=\gamma $} -- (4.5,4);
 	\draw (4.5,3.4) node[right] {$ a/r-b/q=\gamma - \epsilon d $} -- (4.5,3.4);
 	\draw (4.5,2.8) node[right] {$ a/r-b/q=\gamma - 2 \epsilon d $} -- (4.5,2.8);
 	\draw (4.5,2.2) node[right] {$ a/r-b/q=\gamma - 3 \epsilon d $} -- (4.5,2.2);
 	
 	\draw[important line] (3,2.5) -- (2,3.5);
 	\draw (2,3.5) node[above] {$= \infty$} -- (2,3.5);
 	
 	\draw[important line] (3.3,1.45) -- (5,1);
 	\draw (5,1) node[right] {$\CC_5^{-1} R^{\epsilon d} \leq \cdot \leq \CC_5 (1 + R^{2 \epsilon d})$} -- (5,1);
 	
 	\draw[important line] (4,0.5) -- (1,2);
 	\draw (1,2) node[above] {$\leq \CC_5$} -- (1,2);

 	\end{tikzpicture}
 	\caption{The behavior of the function $\cc(p, q, r, \TT)$.}
 \end{figure}
 
 \begin{proof}
 For each $n \in \NN$ let $\SSS_n$ be the test space $\SSS_{(p, \lambda, a, b, \kappa)}$ from Corollary~\ref{C1} with $p$, $a$, and $b$ as above, $\kappa = R^n$, and $\lambda = R^{- n \gamma + (n+2)\epsilon d}$. We let $\TT$ be the space obtained by using Proposition~\ref{P1} for the sequence $\{ \SSS_n : n \in \NN\}$. We have the following estimates: if
 $a/r - b/q = \gamma$, then
 \begin{displaymath}
 \cc(p,q,r,\TT) \geq \frac{1}{\CC \CC_4} \lim_{n \rightarrow \infty} R^{- n \gamma + (n+2)\epsilon d} R^{n \gamma} = \infty,
 \end{displaymath}  
 if $a/r - b/q \in (\gamma - 2 \epsilon d, \gamma - \epsilon d)$, then
 \begin{displaymath}
 \cc(p,q,r,\TT) \geq \frac{1}{\CC \CC_4} \sup_{n \in \NN} R^{- n \gamma + (n+2)\epsilon d} R^{n (\gamma - 2 \epsilon d)} = \frac{R^{\epsilon d}}{\CC \CC_4}
 \end{displaymath}
 and
 \begin{displaymath}
 \cc(p,q,r,\TT) \leq \CC \CC_4 \sup_{n \in \NN} \Big( 1 + R^{- n \gamma + (n+2)\epsilon d} R^{n (\gamma - \epsilon d)} \Big) \leq \CC \CC_4 \Big( 1 + R^{2\epsilon d} \Big),
 \end{displaymath}
 and, if $a/r - b/q \leq \gamma - 3 \epsilon d$, then
 \begin{displaymath}
 \cc(p,q,r,\TT) \leq \CC \CC_4 \sup_{n \in \NN} \Big( 1 + R^{- n \gamma + (n+2)\epsilon d} R^{n (\gamma - 3 \epsilon d)} \Big) = 2 \CC \CC_4.
 \end{displaymath}
 Therefore, $\TT$ satisfies the desired properties. 
 \end{proof}
 
 At the end of this section we present another result for composite test spaces, which is particularly helpful if the domain of $F$ in Theorem~\ref{T1} is of the form $(\delta, 1]$, or if the domain is of the form $[\delta,1]$, but either $\delta = 1$ or $F$ is not continuous at $\delta$. 
 
% \begin{lemma} \label{L6}
% 	Let $p \in (1, \infty)$ and $\delta \in [0,1)$. Then there exists a composite test space $\widetilde{\TT} = \widetilde{\TT}_{p, \delta}$ such that $\cc(p,q,r, \widetilde{\TT}) < \infty$ if and only if $1/q > \delta$ and $r \geq q$.	
% \end{lemma}
% 
% \begin{proof}
% 	Fix $p \in (1, \infty)$ and $\delta \in [0,1)$. For each $n \in \NN$ take $\gamma_n \in \RR$ and $a_n, b_n \in \NN$ such that $a_n / r - b_n \delta \in (\gamma_n - 2d_n/n, \gamma_n - d_n/n)$ for each $r \in [1, \infty]$, where $d_n = \sqrt{a_n^2 + b_n^2}$. Let $\TT_n$ be the composite test space $\TT$ from Lemma~\ref{L5} with $p$ as above, $\gamma = \gamma_n$, $a = a_n$, $b = b_n$, $R=n^n$, and $\epsilon = 1/n$. Then it is easy to check that $\widetilde{\TT}$ may be chosen to be the space obtained by using Proposition~\ref{P2} for the sequence of composite test spaces $\{ \TT_n : n \in \NN\}$ (to obtain $\cc(p,q,r, \widetilde{\TT}) = \infty$ for $1/q > \delta$ and $r < q$ we use Remark~\ref{R1}, see Section~\ref{S6}). 
% \end{proof}

\begin{lemma} \label{L6}
	Let $p \in (1, \infty)$, $\delta \in [0,1]$, and $\omega \in [0,\delta]$. Then there exists 
	\begin{itemize}
		\item a composite test space $\TT^\leq = \TT^\leq_{p, \delta, \omega}$ such that $\cc(p,q,r, \TT^\leq) < \infty$ if and only if $1/q > \delta$, $r \geq q$ or $1/q = \delta$, $1/r \leq \omega$,
		\item a composite test space $\TT^< = \TT^<_{p, \delta, \omega}$ such that $\cc(p,q,r, \TT^<) < \infty$ if and only if $1/q > \delta$, $r \geq q$ or $1/q = \delta$, $1/r < \omega$.
	\end{itemize}	
\end{lemma}

\begin{proof}
	Fix $p \in (1, \infty)$, $\delta \in [0,1]$, and $\omega \in [0,\delta]$. First we construct $\TT^\leq$. For each $n$ take $a_n = n$, $b_n = n^2$, and $\gamma_n$ satisfying $a_n \omega - b_n \delta = \gamma_n - 3 d_n \epsilon_n$, where $d_n = \sqrt{a_n^2 + b_n^2}$ and $\epsilon_n = 1/(3n)$. Let $\TT_n$ be the composite test space $\TT$ from Lemma~\ref{L5} with $p$ as above, $\gamma = \gamma_n$, $a = a_n$, $b = b_n$, $R=n^n$, and $\epsilon = \epsilon_n$. Since $\lim_{n \rightarrow \infty} b_n/a_n = \infty$ and $a_n (\omega+1/n) - b_n \delta > \gamma_n - 2 d \epsilon_n$, it is easy to check that $\TT^\leq$ may be chosen to be the space obtained by using Proposition~\ref{P2} for the sequence of composite test spaces $\{ \TT_n : n \in \NN\}$ (to obtain $\cc(p,q,r, \widetilde{\TT}) = \infty$ for $1/q > \delta$, $r < q$ we use Remark~\ref{R1}, see Section~\ref{S6}). Finally, in order to construct $\TT^<$ we take $a_n = n$, $b_n = n^2$, and $\gamma_n$ satisfying $a_n (\omega-1/n) - b_n \delta = \gamma_n - 3 d_n \epsilon_n$ and $a_n \omega - b_n \delta \in (\gamma_n - 2 d_n \epsilon_n, \gamma_n - d_n \epsilon_n)$, where $d_n$ and $\epsilon_n$ are as before, and then we repeat the previous procedure.
\end{proof}

 \noindent We note that Lemma~\ref{L6} may also be used to construct $\XX$ such that $\Omega^p_{\rm HL}(\XX) = \emptyset$. Indeed, it suffices to take $\TT^<$ with $p$ as above, $\delta = 1$, and $\omega = 0$. 
 
 \section{Proof of Theorem~\ref{T1}} \label{S5}
 
 \noindent \bf Case 1: \rm $F \colon [\delta, 1] \rightarrow [0,1]$, $F$ is continuous at $\delta$. Fix $p \in (1, \infty)$ and $\delta \in [0, 1]$, and take $F \colon [\delta, 1] \rightarrow [0,1]$ concave, non-decreasing, continuous at $\delta$, and such that $F(u) \leq u$ for each $u \in [\delta, 1]$. 
 
 First we construct $\YY$. We can assume that $\delta < 1$, since the case $\delta = 1$ is covered by Lemma~\ref{L6}. Consider the countable set 
 \begin{displaymath}
 \Bigg\{ \Big(\frac{1}{q}, \frac{1}{r}\Big) \in \big( [0,1] \cap \QQ \big)^2 : \Bigg( {\frac{1}{q}} \geq  \delta \ \wedge \ \frac{1}{r} > F\Big(\frac{1}{q}\Big) \Bigg) \ \vee \ \Bigg( \frac{1}{q} < \delta \Bigg) \Bigg\} 
 \end{displaymath}
 and enumerate it to obtain a sequence $\{ P_1, P_2, \dots \}$. Fix $n \in \NN$ and let $P_n = \big( \frac{1}{q_n}, \frac{1}{r_n} \big)$. Since $F$ is concave and non-decreasing, we can choose $\gamma_n \in \RR$, $a_n, b_n \in \NN$, and $\epsilon_n > 0$ such that
 \begin{itemize}
 \item $a_n / r_n - b_n / q_n = \gamma_n$,
 \item if $a_n / r - b_n / q > \gamma_n - 3 \epsilon_n d_n$, then $\frac{1}{q} \geq  \delta$, $\frac{1}{r} > F\big(\frac{1}{q}\big)$ or $\frac{1}{q} < \delta$, where $d_n = \sqrt{a_n^2 + b_n^2}$.
 \end{itemize}
 Let $\TT_n$ be the composite test space $\TT$ from Lemma~\ref{L5} with $p$ as above, $\gamma = \gamma_n$, $a = a_n$, $b = b_n$, $R = 1$, and $\epsilon = \epsilon_n$. It is easy to check that $\YY$ may be chosen to be the space obtained by using Proposition~\ref{P2} for the sequence of composite test spaces $\{ \TT_n : n \in \NN \}$.
 
 Now we construct $\ZZZ$. Again we assume that $\delta < 1$, since the case $\delta = 1$ is covered by Lemma~\ref{L6}. For each $n \in \NN$ and $u \in [\delta, 1]$ we choose $\gamma_{n,u} \in \RR$ and $a_{n,u}, b_{n,u} \in \NN$ such that
 \begin{itemize}
 	\item $\gamma_{n,u} - 2 d_{n,u} / n < a_{n,u} u - b_{n,u} F(u) < \gamma_{n,u} - d_{n,u} / n$, where $d_{n,u} = \sqrt{a_{n,u}^2 + b_{n,u}^2}$,
 	\item if $a_{n,u} / r - b_{n,u} / q \geq \gamma_{n,u} - d_{n,u} / n$, then $\frac{1}{q} \geq \delta$, $\frac{1}{r} > F\big(\frac{1}{q}\big)$ or $\frac{1}{q} < \delta$. 
 \end{itemize}
 Let $\TT_{n,u}$ be the composite test space $\TT$ from Lemma~\ref{L5} with $p$ as above, $\gamma = \gamma_{n,u}$, $a = a_{n,u}$, $b = b_{n,u}$, $R = n^n$, and $\epsilon = 1/n$. Fix $n \in \NN$ and observe that for each $u \in [\delta, 1]$ the set
 \begin{displaymath}
 E_{n,u} = \Big\{ v \in [\delta, 1] : \gamma_{n,u} - 2 d_n / n < av - bF\ref{v} < \gamma_{n,u} - d_n/n \Big \}
 \end{displaymath} 
 is open in $[\delta, 1]$ with its natural topology. Thus $\{ E_{n,u} : u \in [\delta, 1] \}$ is an open cover of $[\delta, 1]$ and we can find a finite subset $U_n \subset [\delta, 1]$ such that $\bigcup_{u \in U_n} E_{n,u} = [\delta, 1]$. Finally, we let $\ZZZ$ be the space obtained by using Proposition~\ref{P2} for the family $\{ \TT_{n,u} : n \in \NN, u \in U_n\}$. We will show that $\ZZZ$ satisfies the desired properties. Fix $u_0 \in [\delta, 1]$ and observe that for each $n \in \NN$ there exists $u_n \in U_n$ such that $u_0 \in E_{n, u_n}$. Therefore, in view of Lemma~\ref{L5},
 \begin{displaymath}
 \cc(p, 1/u_0, 1/ F(u_0), \ZZZ) \geq \frac{1}{\CC^2} \cc(p, 1/u_0, 1/F(u_0), \TT_{n,u_n}) \geq \frac{1}{\CC^2} n^{d_{n,u}}.
 \end{displaymath}
 Since $n$ is arbitrary and $d_{n,u} \geq 1$, we conclude that  $\cc(p, 1/u_0, 1/ F(u_0), \ZZZ) = \infty$ and, as a~result, we obtain $ \cc(p, q, r, \ZZZ) = \infty$ if $\frac{1}{q} \geq \delta$, $\frac{1}{r} \geq F\big(\frac{1}{q}\big)$ or $\frac{1}{q} < \delta$. Now let us consider a~pair $(q,r)$ satisfying $\frac{1}{q} \geq \delta$, $\frac{1}{r} < F\big(\frac{1}{q}\big)$. Then we have
 \begin{displaymath}
 d(q,r,F) := \min \Bigg\{ d_e\Bigg( \Big( \frac{1}{q}, \frac{1}{r} \Big), \Big(u, F(u) \Big) \Bigg) : u \in [\delta, 1] \Bigg\} > 0,
 \end{displaymath}
 where $d_e$ is the standard Euclidean metric on the plane. Observe that for each $n \in \NN$ and $u \in U_n$ we have the following implication
 \begin{displaymath}
 a_{n,u} / r - b_{n,u} / q > \gamma_{n,u} - 3 d_{n,u} / n \implies d(q,r,F) \leq 2/n.
 \end{displaymath}
 Hence if $n > 2 / d(q,r,F)$, then for each $u \in U_n$ we have $a_{n,u} / r - b_{n,u} / q \leq \gamma_{n,u} - 3 d_{n,u} / n$, which implies $\cc(p,q,r,\TT_{n,u}) \leq \CC_5$. Finally, since for each of the finitely many pairs $(n,u)$ satisfying $n \leq 2 / d(q,r,F)$ and $u \in U_n$ there is $\cc(p,q,r,\TT_{n,t}) < \infty$, we conclude that $\cc(p,q,r, \ZZZ) < \infty$.
 
 \smallskip \noindent \bf Case 2: \rm $F \colon [\delta, 1] \rightarrow [0,1]$, $F$ is not continuous at $\delta$. Fix $p \in (1, \infty)$ and $\delta \in (0, 1)$, and take $F \colon [\delta, 1] \rightarrow [0,1]$ concave, non-decreasing, satisfying $F(\delta) = \omega < \lim_{u \rightarrow \delta} F(u)$ for some $\omega \in [0,\delta)$, and such that $F(u) \leq u$ for each $u \in [\delta, 1]$. Let $\widetilde{F}$ be the continuous modification of $F$, that is, $\widetilde{F}(u)=F(u)$ for $u \in (\delta, 1)$ and $\widetilde{F}(\delta) = \lim_{u \rightarrow \delta} F(u)$. Then $\widetilde{F}$ satisfies the conditions specified in Case 1. Let $\widetilde{\YY}$ (resp. $\widetilde{\ZZZ}$) be the space obtained in Case 1 for $\widetilde{F}$. We also let $\widetilde{\TT}$ be the composite test space $\TT^\leq$ (resp. $\TT^<$) from Lemma~\ref{L6} with $p$, $\delta$, and $\omega$ as above. It is easy to check that $\YY$ (resp. $\ZZZ$) may be chosen to be the space obtained by using Proposition~\ref{P2} for $\widetilde{\YY}$ (resp. $\widetilde{\ZZZ}$) and countably many copies of $\widetilde{\TT}$.

 \smallskip \noindent \bf Case 3: \rm $F \colon (\delta, 1] \rightarrow [0,1]$. Fix $p \in (1, \infty)$ and $\delta \in [0,1)$, and take $F \colon (\delta, 1] \rightarrow [0,1]$ concave, non-decreasing and such that $F(u) \leq u$ for each $u \in (\delta,1]$. We extend $F$ to $\widetilde{F} \colon [\delta, 1] \rightarrow [0,1]$, setting $\widetilde{F}(\delta) = \lim_{u \rightarrow \delta} F(u)$. Then $\widetilde{F}$ satisfies the conditions specified in Case 1. Let $\widetilde{\YY}$ and $\widetilde{\ZZZ}$ be the spaces obtained in Case 1 for $\widetilde{F}$. We also let $\widetilde{\TT}$ be the composite test space $\TT^<$ from Lemma~\ref{L6} with $p$ and $\delta$ as above, and $\omega = 0$. It is easy to check that $\YY$ (resp. $\ZZZ$) may be chosen to be the space obtained by using Proposition~\ref{P2} for $\widetilde{\YY}$ (resp. $\widetilde{\ZZZ}$) and countably many copies of $\widetilde{\TT}$.
 
\section{Necessary conditions} \label{S6}
 
%\section{Final comments} \label{S6}

In the last section we briefly discuss why there are no alternatives for the shape of $\Omega^p_{\rm HL}(\XX)$ other than those mentioned in Theorem~\ref{T0}. We begin with the following simple observation.

\begin{remark} \label{R0}
	Fix $p \in (1, \infty)$ and let $\XX$ be an arbitrary metric measure space. If $(u,w) \in \Omega^p_{\rm HL}(\XX)$, then $[u,1] \times [w,1] \subset \Omega^p_{\rm HL}(\XX)$.
\end{remark}

\noindent Indeed, this follows by the fact that the Lorentz spaces $L^{p,q}(\XX)$ increase as the parameter $q$ increases.

\smallskip \noindent By Remark~\ref{R0} we know that either $\Omega^p_{\rm HL}(\XX)$ is empty or it consists of points lying under the graph of some non-decreasing function, say $F$, and the domain of $F$ is of the form $[\delta, 1]$ or $(\delta, 1]$ for some $\delta \in [0,1]$ or $\delta \in [0,1)$, respectively. More precisely, for each $u$ from the domain of $F$ we have $(u,w) \in \Omega^p_{\rm HL}(\XX)$ for $w < F(u)$ and $(u,w) \notin \Omega^p_{\rm HL}(\XX)$ for $w > F(u)$ (here we do not focus on whether $(u,F(u))$ belongs to $\Omega^p_{\rm HL}(\XX)$ or not, except for the case $F(u) = 0$, which forces that the first option actually takes place).   

Remark~\ref{R1} below, in turn, explains why the assumption $F(u) \leq u$ is needed. In what follows by $\rm{supp}(\mu)$ we mean the {\it support} of a given measure $\mu$, that is, the set
$$
{\rm supp}(\mu):= \{ x \in X : \mu(B(x,s)) > 0 \ {\rm for} \ {\rm all} \ s>0\}.
$$  

\begin{remark} \label{R1}
	Let $\XX$ be a metric measure space. Assume that there exists an infinite family $\BB$ of pairwise disjoint balls $B$ satisfying $0 < \mu(B) < \infty$. Then for each $p \in (1, \infty)$ we have $\Omega^p_{\rm HL}(\XX) \subset \{(u,w) \in [0,1]^2 : u \leq w \}$.    
\end{remark}

\noindent Indeed, fix $p \in (1, \infty)$ and $1 \leq r < q < \infty$ (the case $q = \infty$ can be considered very similarly). Let $n_0 \in \NN$. We can find a sequence of pairwise disjoint sets $\{E_n : n = 1, \dots, n_0\}$ with the following properties:
\begin{itemize}
	\item each $E_n$ is a union of finitely many elements from $\BB$,
	\item $\mu(E_n) \geq \mu(E_1 \cup \dots \cup E_{n-1})$ for each $n \in \{2, \dots, n_0\}$. 
\end{itemize}
Consider $g_{n_0} \in L^{p,q}(\XX)$ defined by
\begin{displaymath}
g_{n_0} := \sum_{n=1}^{n_0} n^{-\frac{2}{q+r}} \mu(E_n)^{-1/p} \chi_{E_n}.
\end{displaymath}
By Lemma~\ref{L1} the following estimates hold 
\begin{displaymath}
\|g_{n_0}\|_{p,q} \leq \CC_1(p,q) \Big(  \frac{p}{q}\Big)^{1/q} \Big( \sum_{n=1}^{n_0} n^{-\frac{2q}{q+r}} \Big)^{1/q}, \quad
\|g_{n_0}\|_{p,r} \geq \frac{1}{\CC_1(p,r)} \Big(  \frac{p}{r}\Big)^{1/r} \Big( \sum_{n=1}^{n_0} n^{-\frac{2r}{q+r}} \Big)^{1/r}.
\end{displaymath}
Observe that for each $x \in \rm{supp}(\mu)$ we have $\MM_\XX g(x) \geq g(x)$. Since $2r/(q+r) < 1 < 2q/(q+r)$, we obtain
$ 
\lim_{n_0 \rightarrow \infty} \frac{\|g_{n_0}\|_{p,r}}{\|g_{n_0}\|_{p,q}} = \infty
$
and, consequently, $(\frac{1}{q}, \frac{1}{r}) \notin \Omega^p_{\rm HL}(\XX)$. 

\smallskip \noindent
One additional comment should be made here. Namely, if $\BB$ from Remark~\ref{R1} does not exist, then there is only finitely many points $x \in \rm{supp}(\mu)$ such that $\mu(B(x,s_x)) < \infty$ for some $s_x > 0$. In this case $\Omega^p_{\rm HL}(\XX) = [0,1]^2$ holds trivially for each $p \in (1, \infty)$.

Finally, the fact that $\Omega^p_{\rm HL}(\XX)$ is convex, and hence $F$ must be concave, is justified by the following interpolation argument. 

\begin{theorem} \label{T2}
	Fix $p \in [1, \infty)$, $1 \leq q_0 \leq q_1 \leq \infty$, and $1 \leq r_0, r_1 \leq \infty$ such that $q_i \leq r_i$ for $i \in \{0, 1\}$. Let $\XX$ be a metric measure space and assume that the associated maximal operator $\MM_\XX$ is bounded from $L^{p, q_i}(\XX)$ to $L^{p, r_i}(\XX)$ for $i \in \{0, 1\}$. Then for each $\theta \in (0,1)$ the operator $\MM_\XX$ is bounded from $L^{p, q_\theta}(\XX)$ to $L^{p, r_\theta}(\XX)$, where
	\begin{displaymath}
	\frac{1}{q_\theta} = \frac{1-\theta}{q_0} + \frac{\theta}{q_1}, \qquad \frac{1}{r_\theta} = \frac{1-\theta}{r_0} + \frac{\theta}{r_1}.
	\end{displaymath}
\end{theorem}

\noindent We explain briefly how Theorem~\ref{T2} can be inferred from the general theory of interpolation. We begin with the comment that Lorentz spaces in this context were considered for the first time by Hunt in \cite{Hu}. However, the theorem formulated there does not cover Theorem~\ref{T2}. Hence, we are forced to refer to the literature where some more advanced interpolation methods are developed. The appropriate variant of Theorem~\ref{T2} for linear operators can be directly deduced from \cite[Theorem 5.3.1]{BL} (see also \cite{Ma}, where the $K$-functional for the couple $(L^{p,q_0}, L^{p,q_1})$ is computed). Then, a suitable linearization argument (see \cite{Ja}, for example) allows us to extend this result to the class of sublinear operators and thus the maximal operator $\MM_\XX$ is also included.   

Although there are several ways to deduce Theorem~\ref{T2} from the theorems that appear in the literature, each of them, to the author's best knowledge, requires a deep understanding of the interpolation theory. As the author found an elegant, elementary proof of Theorem~\ref{T2}, he decided to present it in the Appendix. 

The last issue we want to mention is the boundary problem. Namely, in Theorem~\ref{T1} we assume that 
$$
\partial\Omega^p_{\rm HL}(\XX) \subset \Omega^p_{\rm HL}(\XX) 
\quad {\rm or} \quad
\partial\Omega^p_{\rm HL}(\XX) \cap \Omega^p_{\rm HL}(\XX) = \emptyset
$$
(here $\partial\Omega^p_{\rm HL}(\XX)$ denotes the boundary of $\Omega^p_{\rm HL}(\XX)$ considered as a subset of $[0,1]^2$). It is natural to ask whether there are any other options except the two mentioned above. In fact, Proposition~\ref{P2} combined with Lemma~\ref{L5} and Lemma~\ref{L6} can provide a wide range of different possibilities. For example, if $F$ from Theorem~\ref{T1} is strictly concave, then for a~given set $E \subset [\delta, 1]$ such that $\overline{E}$ is countable we can find $\XX$ such that $\MM_\XX$ is bounded from $L^{p, 1/u}(\XX)$ to $L^{p, 1/F(u)}(\XX)$ if and only if $u \notin E$. Nevertheless, it is probably very difficult to describe precisely all forms that the intersections $\partial\Omega^p_{\rm HL}(\XX) \cap \Omega^p_{\rm HL}(\XX)$ can take.

\section*{Appendix. Proof of Theorem~\ref{T2}} \label{SA}

Here we give an elementary proof of Theorem~\ref{T2}. In what follows the operator is specified to be $\MM_\XX$, but one can also replace it with, for example, any operator $\HH$ satisfying the following assertions:
\begin{enumerate}[label=(\Alph*)]
	\item $0 \leq f_1 \leq f_2 \implies 0 \leq \HH f_1 \leq \HH f_2$,
	\item $|\HH f| \leq \HH |f|$,
	\item $\HH(|f_1| + |f_2|) \leq C_\HH (\HH |f_1| + \HH |f_2|)$.
\end{enumerate}   

	First we observe that it suffices to consider the case $q_0 < q_1$ and $r_0 < r_1$. Indeed, in each of the remaining cases the thesis is an easy consequence of Fact~\ref{F3}.
	Fix $\theta \in (0,1)$ and let $\CC_\rightarrow$ be such that
	\begin{displaymath}
	\|\MM_\XX g \|_{p, r_i} \leq \CC_\rightarrow \| g \|_{p, q_i}, \qquad g \in L^{p, q_i}(\XX), \ i \in \{0,1\}.
	\end{displaymath}
	Our aim is to obtain the inequality
	\begin{equation}\label{A1}
	\|\MM_\XX g \|_{p, r_\theta} \leq C \, \| g \|_{p, q_\theta},
	\end{equation}
	for each $g \in L^{p, q_\theta}(\XX)$ with some $C$ independent of $g$.
	For any measurable function $g \colon X \rightarrow \CCC$ we introduce $\SSSS g, \TTTT g \colon \ZZ \rightarrow [0, \infty]$ given by
	\begin{displaymath}
	\SSSS g(n) := 2^n d_g(2^n)^{1/p}, \qquad n \in \ZZ, 
	\end{displaymath}
	and 
	\begin{displaymath}
	\TTTT g(n) := \SSSS\MM_\XX g(n) = 2^n d_{\MM_\XX g}(2^n)^{1/p}, \qquad n \in \ZZ.
	\end{displaymath}
	We observe that for each $q \in [1, \infty]$ there is a numerical constant $\CC_\square(p,q)$ such that
	\begin{displaymath}
	\frac{1}{\CC_\square(p,q)} \, \| \SSSS g \|_{q} \leq \|g\|_{p,q} \leq \CC_\square(p,q) \, \| \SSSS g \|_{q}, \qquad g \in L^{p, q}(\XX),
	\end{displaymath}
	where $\|  \cdot  \|_q$ denotes the standard norm on $\ell^q(\ZZ)$. Let
	\begin{displaymath}
	\CC_\square = \max \{\CC_\square(p,q_0), \CC_\square(p,q_\theta), \CC_\square(p,q_1), \CC_\square(p,r_0), \CC_\square(p,r_\theta), \CC_\square(p,r_1)\}.
	\end{displaymath}
	Thus for each $i \in \{ 0, 1\}$ we have
	\begin{equation}\label{A2}
	\| \TTTT g \|_{r_i} \leq \CC_\square^2 \, \CC_\rightarrow \, \|\SSSS g \|_{q_i}
	\end{equation}
	and our aim is to obtain the inequality
	\begin{equation}\label{A3}
	\| \TTTT g \|_{r_\theta} \leq C' \, \|\SSSS g \|_{q_\theta},
	\end{equation}
	which would imply \eqref{A1} with $C = C' \, \CC_\square^2$. 
	
	In order to deduce \eqref{A3} from \eqref{A2} we follow the classical proof of Marcinkiewicz interpolation theorem for operators acting on Lebesgue spaces (see \cite[Theorem 1]{Z}). It turns out that this strategy can be successfully applied but we must take into account certain additional difficulties. Namely, our ``operator'' should be the transformation $\SSSS g \mapsto \TTTT g$. However, this operation cannot be considered as a well defined operator, because there are usually many different functions with the same distribution function. Thus, we proceed with the details.
	
	Assume that $r_1 < \infty$ and fix $f \in L^{p,q_\theta}(\XX)$ satisfying $f \geq 0$. For each $\lambda > 0$ we introduce the set $N_\lambda = \{n \in \ZZ : \SSSS f > \lambda \}$. Observe that either $N_\lambda = \emptyset$ or $N_\lambda$ consists of finitely many elements $n_1 > \ldots > n_m$, $m \in \NN$. For each $j \in \ZZ$ let $E_j = \{x \in X  : f(x) \geq 2^j \}$. If $N_\lambda = \emptyset$, then we let $f_0^\lambda = 0$ and $f_1^\lambda = f$. Otherwise, if $N_\lambda \neq \emptyset$, we define
	\begin{displaymath}
	f_0^\lambda = f \cdot \big( \chi_{E_{n_1}} + \sum_{k=2}^{m} \chi_{E_{n_k} \setminus E_{n_{k-1}}} \big), \qquad f_1^\lambda = f \cdot \sum_{j \in \ZZ \setminus N_\lambda} \chi_{E_{n_j} \setminus E_{n_{j-1}}}.
	\end{displaymath}
	Notice that $f \leq f_0^\lambda + f_1^\lambda$ and hence $\MM_\XX f \leq \MM_\XX f_0^\lambda + \MM_\XX f_1^\lambda$. Moreover, we have
	\begin{displaymath}
	\SSSS f_0^\lambda(n) = \SSSS f(n) > \lambda, \qquad n \in N_\lambda,
	\end{displaymath}
	and
	\begin{displaymath}
	\SSSS f_0^\lambda(n) \leq \min\{ \lambda, \SSSS f(n)\}, \qquad n \in \ZZ.
	\end{displaymath}
	Let $(\SSSS f)^\lambda_0 = \SSSS f \cdot \chi_{N_\lambda}$ and $(\SSSS f)^\lambda_1 = \SSSS f \cdot \chi_{\ZZ \setminus N_\lambda}$. Then it is not hard to check that
	\begin{equation}\label{geom}
	\| \SSSS f^\lambda_i \|_{q_i} \leq \big( 1 + 2^{-q_i/p} + 4^{-q_i/p} + \ldots \big)^{1/q_i} \, \|(\SSSS f)^\lambda_i \|_{q_i}.
	\end{equation} 
	
	Next we study the distribution functions of $(\SSSS f)^\lambda_i$, $i \in \{0,1\}$, more carefully. Observe that we have $d_{(\SSSS f)^\lambda_0}(y) \leq d_{\SSSS f}(\lambda)$ for $0 < y < \lambda$ and $d_{(\SSSS f)^\lambda_0}(y) \leq d_{\SSSS f}(y)$ for $y \geq \lambda$. Hence, combining the above estimates and the fact that $d_{(\SSSS f)^\lambda_0}$ is non-increasing with the equality
	\begin{displaymath}
	2^{q_0} \int_{0}^{\lambda/2} y^{q_0-1} \, {\rm d}y = \int_{0}^{\lambda} y^{q_0-1} \, {\rm d}y,
	\end{displaymath}
	we conclude that
	\begin{equation}\label{d0}
	\int_{0}^\infty y^{q_0-1} d_{(\SSSS f)^\lambda_0}(y) \, {\rm d}y \leq \frac{2^{q_0}}{2^{q_0}-1} \int_{\lambda/2}^\infty y^{q_0-1} d_{\SSSS f}(y) \, {\rm d}y \leq 2^{q_0} \int_{\lambda/4}^\infty (y - \lambda/4)^{q_0-1} d_{\SSSS f}(y) \, {\rm d}y.
	\end{equation}
	Similarly, we note that $d_{(\SSSS f)^\lambda_1}(y) \leq d_{\SSSS f}(y)$ for $0 < y < \lambda$ and $d_{(\SSSS f)^\lambda_0}(y) = 0$ for $y \geq \lambda$, which gives
	\begin{equation}\label{d1}
	\int_{0}^\infty y^{q_1-1} d_{(\SSSS f)^\lambda_1}(y) \, {\rm d}y \leq \int_{0}^\lambda y^{q_0-1} d_{\SSSS f}(y) \, {\rm d}y \leq 2^{2q_0} \int_{0}^{\lambda/4} y^{q_0-1} d_{\SSSS f}(y) \, {\rm d}y.
	\end{equation}
	
	Now we turn our attention to $\TTTT f$. Fix $y > 0$ and $\lambda = \lambda(y)$ (which will be specified later on), and notice that $\MM_\XX f \leq \MM_\XX f_0^\lambda + \MM_\XX f_1^\lambda$ implies $\TTTT f(n) \leq 2^{1/p} (\TTTT f_0^\lambda(n-1) + \TTTT f_1^\lambda(n-1))$ for each $n \in \NN$. Hence
	\begin{equation}\label{d2}
	d_{\TTTT f}(y) \leq d_{\TTTT f_0^\lambda}(y/2^{1/p}) + d_{\TTTT f_1^\lambda}(y/2^{1/p}).
	\end{equation}
	By the hypothesis we have
	\begin{equation}\label{d3}
	d_{\TTTT f_i^\lambda}(y/2^{1/p}) \leq 2^{r_i/p} \, \frac{\|\TTTT f_i^\lambda \|_{r_i}^{r_i}}{y^{r_i}} \leq \big( 2^{1/p} \CC_\square^2 \CC_\rightarrow \big)^{r_i} \, \frac{\|\SSSS f_i^\lambda \|_{q_i}^{r_i}}{y^{r_i}}. 
	\end{equation}
	Therefore, combining \eqref{geom}, \eqref{d0}, \eqref{d1}, \eqref{d2}, and \eqref{d3} gives
	\begin{align*}
	\| \TTTT f \|_{r_\theta}^{r_\theta} & = r_\theta \int_{0}^\infty y^{r_\theta - 1} d_{\TTTT f}(y) \, {\rm d}y \\ & \leq \widetilde{C} \, \Big( \int_0^\infty y^{r_\theta-r_0-1} \, \Big( \int_{\lambda(y)/4}^\infty (t-\lambda(y)/4)^{q_0 - 1} d_{\SSSS f}(t) \, {\rm d}t \Big)^{r_0/q_0} \, {\rm d}y \\ &
	+ \int_0^\infty y^{r_\theta-r_1-1} \, \Big( \int_0^{\lambda(y)/4} t^{q_1 - 1} d_{\SSSS f}(t) \, {\rm d}t \Big)^{r_1/q_1} \, {\rm d}y
	\Big),
	\end{align*}
	with some constant $\widetilde{C}$ which may depend on $p, q_0, q_1, r_0, r_1, \theta$, and $\CC_\rightarrow$ but is independent of $f$ and the choice of a suitable function $\lambda = \lambda(y)$.
	
	It is worth noting here that the inequality above reduces the problem to estimating the expression of the form very similar to that appearing in \cite[(3.7)]{Z} (here $d_{\SSSS f}$, $\lambda/4$, $q_0$, $q_1$, $r_0$, $r_1$, and $r_\theta$ play the roles of $m$, $z$, $a_2$, $a_1$, $b_2$, $b_1$, and $b$, respectively). Thus, in order to obtain \eqref{A3}, we may repeat the remaining calculations without any further changes. We briefly sketch the rest of the proof for the sake of completeness.
	
	Denote by $P$ and $Q$ the two double integrals in the last estimate. Then
	\begin{displaymath}
	P^{q_0/r_0} = \sup_{\omega_0} \, \int_0^\infty y^{r_\theta - r_0 - 1} \, \int_{\lambda(y) / 4}^{\infty} (t - \lambda(y)/4)^{q_0-1} \, d_{\SSSS f}(t) \, {\rm d}t \, \omega_0(y) \, {\rm d}y
	\end{displaymath}
	and
	\begin{displaymath}
	Q^{q_1/r_1} = \sup_{\omega_1} \, \int_0^\infty y^{r_\theta - r_1 - 1} \, \int_{0}^{\lambda(y)/4} t^{q_1-1} \, d_{\SSSS f}(t) \, {\rm d}t \, \omega_1(y) \, {\rm d}y,
	\end{displaymath}
	where the functions $\omega_i \geq 0$ satisfy
	\begin{displaymath}
	\int_0^\infty y^{r_\theta-r_i-1} \omega_i^{(r_i/q_i)'}(y) \, {\rm d}y \leq 1,
	\end{displaymath} 
	with $(r_i/q_i)'$, the exponent conjugate to $r_i/q_i$. We set $\lambda(y) = 4 \|\SSSS f\|_{q_\theta}^{- \tau \xi} y^\xi$, where $\tau$ and $\xi$ will be determined later on. Now, by using H\"older's inequality, we obtain
	\begin{align*}
	& \int_0^\infty y^{r_\theta-r_0-1} \int_{\|\SSSS f\|_{q_\theta}^{- \tau \xi} y^\xi}^{\infty} (t - \|\SSSS f\|_{q_\theta}^{- \tau \xi} y^\xi)^{q_0-1} \, d_{\SSSS f}(t) \, {\rm d}t \, \omega_0(y) \, {\rm d}y \\
	& \quad \leq \int_0^\infty t^{q_0-1} \, d_{\SSSS f}(t) \int_{0}^{\| \SSSS f \|_{q_\theta}^{\tau} t^{\frac{1}{\xi}}} y^{r_\theta-r_0-1} \omega_0(y) \, {\rm d}y \, {\rm d}t \\
	& \quad \leq \int_0^\infty t^{q_0-1} \, d_{\SSSS f}(t) \Big( \int_{0}^{\| \SSSS f \|_{q_\theta}^{\tau} t^{\frac{1}{\xi}}} y^{r_\theta-r_0-1} {\rm d}y \Big)^\frac{q_0}{r_0} \, \Big( \int_{0}^{\| \SSSS f \|_{q_\theta}^{\tau} t^{\frac{1}{\xi}}} y^{r_\theta-r_0-1} \omega_0^{(\frac{r_0}{q_0})'}(y) {\rm d}y \Big)^\frac{1}{(r_0/q_0)'} {\rm d}t \\
	& \quad \leq (r_\theta-r_0)^{-q_0/r_0} \| \SSSS f \|_{q_\theta}^{\frac{(r_\theta-r_0) q_0 \tau}{r_0}} \int_0^\infty t^{q_0-1 + \frac{(r_\theta-r_0) q_0}{\xi r_0}} \, d_{\SSSS f}(t) \, {\rm d}t.
	\end{align*}
	Similarly, we obtain
	\begin{align*}
	& \int_0^\infty y^{r_\theta-r_1-1} \int_0^{\|\SSSS f\|_{q_\theta}^{- \tau \xi} y^\xi} t^{q_1-1} \, d_{\SSSS f}(t) \, {\rm d}t \, \omega_1(y) \, {\rm d}y \\
	& \qquad \leq (r_1-r_\theta)^{-q_1/r_1} \| \SSSS f \|_{q_\theta}^{\frac{(r_\theta-r_1) q_1 \tau}{r_1}} \int_0^\infty t^{q_1-1 + \frac{(r_\theta-r_1) q_1}{\xi r_1}} \, d_{\SSSS f}(t) \, {\rm d}t.
	\end{align*}
	Collecting these results we conclude that
	\begin{displaymath}
	\| \TTTT f \|_{r_\theta}^{r_\theta} \leq \widetilde{C}' \, \sum_{i=0}^1 \| \SSSS f \|_{q_\theta}^{(r_\theta-r_i) \tau} \Big( \int_0^\infty t^{q_i - 1 + \frac{(r_\theta - r_i) q_i}{r_i \xi}} d_{\SSSS f}(t) \, {\rm d}t \Big)^{r_i/q_i},
	\end{displaymath}
	for some $\widetilde{C}'$ independent of $f$. Choosing
	\begin{equation}\label{tau}
	\tau = \frac{q_\theta (r_1/q_1 - r_0/q_0) }{r_1 - r_0}, \quad \xi = \frac{q_\theta^{-1} (r_1^{-1} - r_\theta^{-1})}{r_\theta^{-1} (q_1^{-1} - q_\theta^{-1})},
	\end{equation}
	gives that both terms in the sum above equal $\| \SSSS f \|_{q_\theta}^{r_\theta}$. Thus \eqref{A3} holds with $C' = (2 \widetilde{C}')^{1/r_\theta}$, which completes the proof in the case $r_1 < \infty$.
	
	Finally, let us assume that $r_1 = \infty$. If $q_1 = \infty$, then the formulas in \eqref{tau} reduce to $\tau = 0$ and $\xi = 1$. We choose $\lambda(y) = cy$ for some sufficiently small constant $c > 0$. In fact, if $c < \CC_\rightarrow^{-1} \CC_\square^{-2} 2^{-1/p}$, then we have $d_{\TTTT f^\lambda_1}(y/2^{1/p}) = 0$, while $d_{\TTTT f^\lambda_0}(y/2^{1/p})$ may be estimated in the same way as it was done before. On the other hand, if $q_1 < \infty$, then the formulas in \eqref{tau} reduce to $\tau = q_\theta / q_1$ and $\xi = q_1 / (q_1 - q_\theta)$. Again, it can be shown that if $\lambda(y) = c' \|f\|_{q_\theta}^{-q_\theta / (q_1 - q_\theta)} y^{q_1/(q_1- q_\theta)}$, where $c'>0$ is sufficiently small (but independent of $f$ and $y$), then $d_{\TTTT f^\lambda_1}(y/2^{1/p}) = 0$ and $d_{\TTTT f^\lambda_0}(y/2^{1/p})$ may be estimated as before. This completes the proof in the case $r_1 = \infty$.


\begin{thebibliography}{999}
		
		\bibitem{Al} J. M. Aldaz, 
		{\it An example on the maximal function associated to a nondoubling measure}, Publ. Mat. {\bf 49(2)} (2005), 453--458.
		
		\bibitem{BS} C. Bennett, R. Sharpley, 
		{\it Interpolation of operators}, Pure Appl. Math. {\bf 129}, 
		Academic Press, Inc., Boston, MA, 1988.
		
		\bibitem{BL} J. Bergh, J. L\"ofstr\"om, {\it Interpolation Spaces: An Introduction}, Grundlehren Math. Wiss. {\bf 223}, Springer-Verlag, Berlin-Heidelberg-New York, 1976.
		
		\bibitem{Hu} R. Hunt, {\it An extension of the Marcinkiewicz interpolation}, Bull. Amer. Math. Soc. {\bf 70} (1964), 803--807.
		
		\bibitem{Ja} S. Janson, {\it On the interpolation of sublinear operators}, Studia Math. {\bf 75} (1982), 51--73.
		
		\bibitem{Ko1} D. Kosz, 
		{\it On relations between weak and strong type inequalities for maximal operators on non-doubling metric measure spaces}, Publ. Mat. {\bf 62(1)} (2018), 37--54.
		
		\bibitem{Ko2} D. Kosz, 
		{\it On relations between weak and restricted weak type inequalities for maximal operators on non-doubling metric measure spaces}, Studia Math. {\bf 241} (2018), 57--70.
		
		\bibitem{Ko3} D. Kosz, 
		{\it Maximal operators on Lorentz spaces in non-doubling setting}, Math. Z. (accepted).
		
		\bibitem{Li1} H.-Q. Li, 
		{\it La fonction maximale de Hardy--Littlewood sur une classe d'espaces métriques mesurables}, C. R. Math. Acad. Sci. Paris {\bf 338} (2004), 31--34.
		
		\bibitem{Li2} H.-Q. Li, 
		{\it La fonction maximale non centrée sur les variétés de type cuspidale}, J. Funct. Anal. {\bf 229} (2005), 155--183.
		
		\bibitem{Li3} H.-Q. Li, 
		{\it Les fonctions maximales de Hardy--Littlewood pour des mesures sur les variétés cuspidales},
		J. Math. Pures Appl. {\bf 88} (2007), 261--275.
		
		\bibitem{Ma} L. Maligranda, {\it The K-functional for symmetric spaces}, Lect. Notes Math. {\bf 1070} (1984), 169--182.
		
		\bibitem{NTV} F. Nazarov, S. Treil and A. Volberg, 
		{\it Weak type estimates and Cotlar inequalities for Calderón--Zygmund operators on nonhomogeneous spaces}, 
		Int. Math. Res. Not. IMRN {\bf 9} (1998), 463--487.
		
		\bibitem{Sa} Y. Sawano, {\it Sharp estimates of the modified Hardy–Littlewood maximal operator on the nonhomogeneous space via covering lemmas}, Hokkaido Math. J. {\bf 34} (2005), 435--458.
		
		\bibitem{Sj} P. Sj\"ogren, 
		{\it A remark on the maximal function for measures in $R^n$}, 
		Amer. J. Math. {\bf 105} (1983), 1231--1233.
		
		\bibitem{St1} K. Stempak, {\it Modified Hardy-Littlewood maximal operators on nondoubling metric measure spaces}, Ann. Acad. Sci. Fenn. Math. {\bf 40} (2015), 443--448.
		
		\bibitem{St2} K. Stempak, {\it Examples of metric measure spaces related to modified Hardy-Littlewood maximal operators}, Ann. Acad. Sci. Fenn. Math. {\bf 41} (2016), 313--314.
		
		\bibitem{Z} A. Zygmund, 
		{\it On a theorem of Marcinkiewicz concerning interpolation of operations}, J. Math. Pures Appl. {\bf 34} (1956), 223--248. 		
	\end{thebibliography}
\end{document}